\documentclass[12pt,a4paper,oneside]{amsart}
\usepackage{amsfonts, amsmath, amssymb, amsthm, amscd}
\usepackage{anysize}
\usepackage{mathtools}
\marginsize{2cm}{2cm}{1.5cm}{1.5cm}
\usepackage[utf8]{inputenc}
\usepackage[english]{babel}
\usepackage{cmap}
\usepackage{enumerate}

\usepackage{xcolor}
\usepackage{hyperref}
\hypersetup{
	colorlinks   = true, %Colours links instead of ugly boxes
	urlcolor     = blue, %Colour for external hyperlinks
	linkcolor    = blue, %Colour of internal links
	citecolor    = red %Colour of citations
}
\newcommand{\Href}[2]{\hyperref[#2]{#1~\ref{#2}}}
\usepackage{anysize}
\marginsize{2cm}{2cm}{1.5cm}{1.5cm}

\makeatletter
\@namedef{subjclassname@2020}{%
  \textup{2020} Mathematics Subject Classification}
\makeatother

\def\R{{\mathbb R}}

\providecommand{\parenth}[1]{\left(#1\right)}%

\numberwithin{equation}{section}

\newtheorem{thm}{Theorem}

\title{Erd\H{o}s--Ko--Rado and Hilton--Milner theorems for two-forms}
\author{Grigory Ivanov
\email{grimivanov@gmail.com}
\and
Seyda K\"{o}se
\email{seyda.kOese@ist.ac.at}
\address{ 
Institute of Science and Technology Austria (IST Austria), 
Klosterneuburg, 3400, Austria}}

\subjclass[2020]{Primary: 05D05, Secondary: 15A75, 	14N20}
\keywords{Erd\H{o}s--Ko--Rado theorem, Hilton--Milner theorem, intersecting family}
\begin{document}
\begin{abstract}
In this short note we show that both  generalizations of celebrated Erd\H{o}s--Ko--Rado theorem and Hilton--Milner theorem to the setting of exterior algebra in the simplest non-trivial case of two-forms follow from the folklore puzzle about possible arrangements of an intersecting family of lines.
\end{abstract}
\maketitle

\section{Introduction}

We recall that a family of sets is called \emph{intersecting} if any two sets of the family  have a non-empty intersection. We assume that $n$ is strictly greater than $4$ throughout the note. 
The celebrated Erd\H{o}s--Ko--Rado theorem  \cite{erdos1961intersection} states

\medskip \noindent
\emph{Suppose that $r \leq n/2$. If $\mathcal{F}$ is
an intersecting family of $r$-element subsets of $\{1, \dots, n\}$,
then $\left|\mathcal{F}\right| \leq {n-1 \choose r-1}$. If more strongly
$r < n/2$, then the equality $\left|\mathcal{F}\right|={n-1 \choose r-1}$
holds only if all the sets in $\mathcal{F}$ share a common element.}

\medskip
\noindent

Hilton and Milner \cite{hilton1967some} obtained the following stability extension of this result: 

\medskip \noindent
\emph{Suppose that $r < n/2$. If $\mathcal{F}$ is
an intersecting family of $r$-element subsets of $\{1, \dots, n\}$ such that there is no common element for all sets of $\mathcal{F}$,
then $\left|\mathcal{F}\right| \leq {n-1 \choose r-1} - {n - r - 1 \choose r-1} + 1$.}

\medskip
\noindent
Except for the case $r = 3,$ there is a unique up to renaming extremal intersecting family 
in the Hilton--Milner result (see \cite{frankl1986non}).

Recently, the Erd\H{o}s--Ko--Rado theorem was extended to the exterior algebra 
(see \cite[Theorem 2.3]{scott2021combinatorics} and \cite[Theorem 1.4]{woodroofe2020algebraic}), where instead of intersecting families of sets, the authors used special subspaces of the space of $r$-forms over $\R^n$ or $\mathbb{C}^n$.
We will say that a subspace $W$ of $\Lambda^{r}\! \parenth{\R^n}$ is \emph{self-annihilating} if $w_1 \wedge w_2 = 0$ for any two $r$-forms $w_1$ and $w_2$ of $W.$
The result may be formulated as follows.

\medskip \noindent
\emph{Suppose that $r < n/2$.
 If $W$ is
 a self-annihilating subspace of $\Lambda^{r}\! \parenth{\R^n},$ 
then $\dim W \leq {n-1 \choose r-1}$.}

\medskip
\noindent
This result implies the inequality of the Erd\H{o}s--Ko--Rado theorem (consider the linear hull of forms $e_{i_1} \wedge \dots \wedge e_{i_r}$ for $\{i_1, \dots i_r\} \in \mathcal{F}$). However, neither the characterization of extremal configurations, nor an extension of the Hilton--Milner theorem were obtained in the setting of exterior algebra. The reasonable conjecture is that $\dim W = {n-1 \choose r-1}$ if and only if all $r$-forms of $W$ are of the form $a \wedge v$ for some fixed $a \in \R^n.$ 

In this note we show that in the simplest non-trivial case of $r=2$ both the characterization of extremal configurations in the Erd\H{o}s--Ko--Rado theorem and the extension of the Hilton--Milner theorem follows from the folklore fact, which we state without proof:

\medskip
\noindent \textbf{Folklore lemma.} 
Let $L$ be a set of lines in  $\mathbb{R}P^{n-1}$  such that any two of them intersect.
Then either all lines pass through one point, or all lines lie in a two-dimensional subspace.

\section{Statement and its proof}
Our small contribution is the following theorem.
\begin{thm}
Suppose $n \geq 5,$ and  $W$ is
 a self-annihilating subspace of $\Lambda^{2}\! \parenth{\R^n}.$
Then
\begin{enumerate}
\item\label{EKR_two-forms_EKR} $\dim W \leq {n-1},$ and the equality holds if and only if  
all $2$-forms of $W$ are of the form $a \wedge v$ for some fixed $a \in \R^n.$ 
\item\label{EKR_two-forms_Hilton-Milner}  if there is no $a \in \R^n$ such that any  $2$-form of $W$ is of the form $a \wedge v,$ then $\dim W \leq 3.$ More strongly, $\dim W = 3$ in this case if and only if 
$W$ is the linear hull of forms $x_1 \wedge x_2, x_2 \wedge x_3, x_3 \wedge x_1$ for some linearly independent $x_1, x_2, x_3  \in \R^n.$
\end{enumerate} 
\end{thm} 
\begin{proof}
Any two-form can be written in the standard form  
$e_1 \wedge e_2 + \dots + e_{2k+1} \wedge e_{2k+2}$ in some basis \cite[Theorem 1.1]{da2008lectures}.
Thus, if a two-form $w$ satisfies $w \wedge w =0$, then it's decomposable, that is,
$w = v_1 \wedge v_2$ for some $v_1, v_2 \in \R^n.$
Consequently, all elements of $W$ are decomposable. 
That is, they correspond to two-dimensional subspaces of $\R^n$ or, equivalently, to lines in $\R P^{n-1}.$ By Folklore lemma, there are two cases:
\begin{enumerate}
\item All the lines pass through one point. Then there are at most $n-1$ linearly independent of them, which easily yields \eqref{EKR_two-forms_EKR}.
\item All the lines belong to some two-dimensional subspace. Then there are at most 3 of them that can be linearly independent, and we have \eqref{EKR_two-forms_Hilton-Milner}.
\end{enumerate}
\end{proof}

\bibliographystyle{alpha}
\bibliography{../../uvolit}
\end{document}